\documentclass[12pt]{article}
\usepackage{geometry}                
\usepackage{graphicx}
\usepackage{amssymb,amsmath,bm}
\usepackage{amsthm}
\usepackage[usenames]{color}
\usepackage{epstopdf}
\renewcommand{\P}{\ensuremath{\mathbb{P}}}
\newcommand{\B}{Bezout}

\def\K{\mathbb{K}}

\def\bez{\mathbf{Bez}}
\def\F{\mathbf{F}}
\def\T{\mathbf{T}}

\def\adots{\mathinner{\mkern2mu\raise1pt\hbox{.}\mkern2mu
\raise4pt\hbox{.}\mkern2mu\raise7pt\hbox{.}\mkern1mu}}

\newcommand \junk[1]{}

\newtheorem{lemma}{Lemma}
\newtheorem{theorem}{Theorem}
\newtheorem{prop}{Proposition}
\newtheorem{defn}{Definition}
\newtheorem{example}{Example}
\newtheorem{cor}{Corollary}

\title{The null space of the Bezout matrix in any basis and gcd's}

\author{Gema Mar\'{\i}a D\'{\i}az-Toca
\footnote{Departamento de Matem\'atica Aplicada, Universidad de  Murcia, Spain. \texttt{gemadiaz@um.es}}\and Mario Fioravanti  \footnote{Departamento de  Matem\'aticas, Estad\'istica y Computaci\'on, Universidad de Cantabria, Spain.  \texttt{mario.fioravanti@unican.es}}}

\date{}

\begin{document}
\maketitle

\begin{abstract}
This manuscript  presents a generalization of the structure of the null space of the Bezout matrix in the monomial basis, see \cite{HeinigRost}, to an arbitrary basis. In addition, two methods for computing the gcd of several polynomials, using also Bezout matrices, without having to convert them to the monomial basis. The main point is that the presented results are expressed with respect to an arbitrary polynomial basis. In recent years, many problems in polynomial systems, stability theory, CAGD, etc., are solved using Bezout matrices in distinct specific bases. Therefore, it is very useful to have results and tools that can be applied to any basis.
\end{abstract}

\section{Introduction}\label{intro}

The Bezout matrix of two polynomials is a symmetric matrix whose generator polynomial was introduced by E. Bezout (1764). 
It was used by  J. J. Sylvester (1853) and C. Hermite (1856) in the context of stability theory.
It was also used by  A. Cayley (1848), who claimed the relation between the resultant and the determinant of the Bezout matrix.
In the late twentieth century, the Bezout matrix regained importance in the context of polynomial algebra, appearing in many books and research articles. One reason for this importance is its relations with some structured matrices, such as Hankel, Toeplitz, and Vandermonde matrices \cite{HeinigRost}, \cite{LT}.
It is a useful tool for the study of the location of zeros of real and complex polynomials, with recent applications to the study of the topology and geometric operations on curves and surfaces, within the area of CAGD (see \cite{DBLP:conf/alenex/BerberichES11}, \cite{D-TFG-VS:2012}, \cite{shakoori:2007}).
The Bezout matrix has many applications in system theory, stability theory of an $n$th order homogeneous linear differential equation, elimination theory, numerical computing and control theory (see \cite{barnett:83}, \cite{HF}).

In recent years, the number of articles on applications of the Bezout matrix to the solution of different problems has increased rapidly. Most of this applications work in a specific polynomial basis, appropriate to the context of the problem. For instance, many many applications of the Bezout matrix in the area of CAGD, use the Bernstein basis. This motivates the interest in the study of the Bezout matrix in this basis (see \cite{BernsteinBez}). In other cases, the data for the problem is given by interpolation values, so it is convenient to use the Lagrange basis and deal with Bezout matrices in this basis (see \cite{Joua}, \cite{Shakoori:BML:2004}, \cite{D-TFG-VS:2012}). The Bezout matrix for Chebyshev polynomials is analyzed in \cite{ChebyBez}. Moreover, it is known that changing from one basis to the monomial basis, and after doing some computations changing back to the initial basis, is a numerically unstable procedure.
Thus, it is good to have a set of theorems that hold for any basis.
The purpose of this paper is to present in a unified manner some relevant results on Bezout matrices. Some of this results are new and others are known, but in all cases, we give the general (arbitrary) basis version.

The paper is organized in the following way. In Section~\ref{SecBez} we recall the general definition of the Bezout matrix and some important properties. In Section~\ref{seccion3} we study the structure of the null space of the Bezout matrix, starting with a result of Heinig and Rost \cite{HeinigRost}, and giving afterwards the general basis version.
In Sections~\ref{barnettsect} and \ref{sectionnullspacesgcd} we present two different methods for computing the gcd of several polynomials. The first one uses a theorem of Barnett \cite{Barnett}, and the second one uses Bezout matrices. Finally, we give some conclusions.

\paragraph{Basic notation:} In this paper, matrices are denoted by bold letters. The transpose of a matrix $\bf{A}$ is denoted by ${\bf A}^T.$ The notation $P(t)\in \mathbb{K}[t]$ indicates a univariate polynomial in the variable $t$ with coefficients from $\mathbb{K}$, where $\mathbb{K}$ is a field of characteristic zero, usually  $\mathbb{Q}$, $\mathbb{R}$ or $\mathbb{C}$.
  The vector space of the polynomials of degree at most
  $n$ is denoted by $\P_{n},$ and $\P_{m,n}$ denotes the space of bivariate polynomials of degrees $m$ and $n$. Finally, $\mathrm{St_n}(t)=\{1,t,\ldots,t^{n-1},t^n\}$ denotes the monomial Basis of $\P_{n}$; when the degree can be omitted, we will simply write $\mathrm{St}(t)$.


 \section{Definition of the Bezout Matrix}\label{SecBez}


The classical definition of the Bezout matrix makes explicit reference to the monomial basis, and it can be found, together with its main properties, in the books \cite{LT}, \cite{mignotte}, \cite{BiniPan:1994}, \cite{FR:1996} and \cite{HeinigRost}, or in the article \cite{HF}. Nevertheless, the definition of the \B\ matrix makes use of the so-called Cayley quotient, which makes no reference to any particular basis in which the given polynomials are represented.
Indeed, in \cite{MFbook, Yang:2001} one can find the following general definition.
\begin{defn}  \label{def:cayleyUnivariate}
  Let $P(t),Q(t)$ be two polynomials with $n=\max\{\deg(P(t)),\deg(Q(t))\}$.
  The Cayley quotient of $P(t)$ and $Q(t)$ is the polynomial
  $C_{p,q}$ of degree at most $n-1$ defined by
  \begin{equation}\label{cayleyUnivariate}
C_{p,q}(t,x)={\displaystyle{\frac{P(t)Q(x)-P(x)Q(t)}{t-x}}}
\end{equation}
Thus, if $\bm{\Phi}(t) = \{\phi_0(t),\dots ,\phi_{n-1}(t)\}$ is a basis for $\P_{n-1}$ then $C_{p,q}$ can be uniquely written as
\begin{equation} \label{BezPhi}
C_{p,q}(t,x)=\sum_{i,j=0}^{n-1} b_{ij} \phi_i(t) \phi_j(x)= \left( \phi_0(t) ,  \ldots,\phi_{n-1}(t)  \right) (b_{ij}) \left(\begin{array}{c} \phi_0(x)\\  \vdots\\ \phi_{n-1}(x) \end{array} \right).
\end{equation}
The symmetric matrix  $\bez_{\bm{\Phi}}(P,Q)=(b_{ij})$ is called the Bezout matrix or the Bezoutian in the polynomial basis $\bm{\Phi}(t).$

\end{defn}
\noindent When the basis is clear or can be omitted, we will denote $\bez(P,Q)$.
Observe that although $C_{p,q}$ is a
rational function of $x$ and $t$, the numerator vanishes if $x=t$;
as such $t-x$ divides $P(t)Q(x)-P(x)Q(t)$ and
the Cayley quotient $C_{p,q}$ is a bivariate polynomial in
$\P_{n-1,n-1}$ as stated.

 \medskip
 Obviously, Bezout matrices associated to different basis are congruent. That is, given two distinct basis of $\P_{n-1}$,  $\bm{\Phi}(t)$ and  $\bm{\Psi}(t)$, and the transformation matrix $\mathbf{P}_{\Psi\rightarrow\Phi}$ between $\bm{\Psi}(t)$ and  $\bm{\Phi}(t)$ such that
$$
\left( \phi_0(t) ,  \ldots,\phi_{n-1}(t)  \right)   \mathbf{P}_{\Psi\rightarrow\Phi} = \left(  \psi_0(t),  \ldots , \psi_{n-1}(t)  \right),
$$
 then $$\bez_\Phi(P,Q)=\mathbf{P}_{\Psi\rightarrow\Phi}  \bez_\Psi(P,Q) \mathbf{P}_{\Psi\rightarrow\Phi}^t .$$

 Some of the most known properties of the Bezout matrix can be found, for example, in \cite{BiniPan:1994} and \cite{HF}

\section{The null space of the Bezout Matrix} \label{seccion3}

In this section, we generalize to an arbitrary basis a known result on the structure of the null space of the Bezout matrix in the monomial basis. 
 Hereafter we assume that the polynomials $P(t)$ and $Q(t)$ are neither null nor proportional; otherwise, the Bezout matrix would be the null matrix, and so its null space would be equal to $\K^n$.

\subsection{The null space of $\mathbf{ {Bez}_{\mathrm{St}}(P,Q)}$}\label{31}

In \cite{HeinigRost}, G. Heinig and K. Rost describe the structure of the null space of the Bezout Matrix in the monomial Basis. The null space of the Bezout matrix has an elegant structure that can be used to determine the common roots of the given polynomials (for a proof see~\cite{HeinigRost}, page 42).

\begin{theorem}\label{kernelmonomial} The null space of $\mathbf{ Bez}_{\mathrm{St}}(P,Q)$ is spanned by the columns of the matrix
\begin{equation}
\left(X_1, X_2, \ldots, X_k \right)
\label{eq:nullspace}
\end{equation}
where each block $X_j$ corresponds to a different common root of
$P(t)$ and $Q(t)$. The dimension of each block is the geometric
multiplicity~$k_j$ of the common root~$x_j$ (i.e., its multiplicity
as a root of the greatest common divisor of $P(t)$ and $Q(t)$).
Moreover each block can be parameterized by the common root~$x_j$
in the form
\begin{equation}\label{eq:paramshape}
X_j=\left( \begin{array}{cccccc}
1 & 0 & 0 & \ldots  & & 0 \\
x_j & 1 & 0 &    \ldots      & & 0 \\
x_j^2 & 2x_j & 2 &\ldots  & & 0 \\
   \vdots & \vdots & \vdots & \vdots& & \vdots \\
   x_j^{n-1} &(n-1)x_j^{n-2} & (n-1)(n-2)x_j^{n-3}&\ldots&& (n-1)^{\underline{k_j-1}}\, x_j^{n-k_j} \\
\end{array}
\right)
\end{equation}
where
$n^{\underline{k_j}} = n(n-1) \cdots(n-k_j+1)$.
\end{theorem}

In particular, when the null space is of dimension $1$, it is generated by
$(1,\alpha,\alpha^2, \ldots, \alpha^{n-1})^T,$ where $\alpha$ is the unique common root of $P(t)$ and $Q(t)$.  Therefore, if  $(v_1,\ldots,v_n)$ is a nonzero vector of the nullspace, then $\alpha=\displaystyle\frac{v_2}{v_1}.$

\subsection{The null space of  $\mathbf{ Bez_\Phi(P,Q)}$} \label{nullspacefi}

This section introduces the generalization of  Theorem \ref{kernelmonomial} to any basis.
Following the notation of Definition \ref{def:cayleyUnivariate}, let $\bm{\Phi}(t) = \{\phi_0(t),\dots ,\phi_{n-1}(t)\}$ be a basis for $\P_{n-1},$ and $\bez_{\bm{\Phi}}(P,Q)$ the \B\ matrix in the polynomial basis $\bm{\Phi}(t)$ (see Equation (\ref{BezPhi})). The following lemma provides the generalization of Theorem \ref{kernelmonomial} to any basis of $\P_{n-1}$
%
%
%
\begin{lemma} \label{lemagral}
Let $i\leq n-1$ and let $\phi_j^{(i)}(t)$ denote the $i$-th derivative of  $\phi_j(t)$ for $j\leq n-1$. Then,
$$
\mathbf{P}^t_{ \Phi \rightarrow \mathrm{St}} \left(\begin{array}{c} 0\\ \vdots \\ 0\\ i! \\ \vdots \\  (n-1)^{\underline{i}}\,t^{n-1-i} \end{array}\right) = \left(\begin{array}{c}\phi_1^{(i)}(t) \\ \vdots \\ \vdots \\ \vdots \\ \phi_n^{(i)}(t)\end{array}\right)
$$
\end{lemma}
\begin{proof} Observe that the entries of the $j$th-row of  $\mathbf{P}^t_{\Phi \rightarrow \mathrm{St}}$ are the coordinates of $\phi_j(t)$ with respect to the monomial basis. Then, if $\mathbf{P}^t_{\Phi \rightarrow \mathrm{St}}=(a_{ij})$, we have
$$
\phi_j(t)= \sum\limits_{k=0}^{k=n-1} a_{jk}t^k.
$$
Hence,
$$
\phi_j^{(i)}(t)= \sum\limits_{k=i+1}^{k=n} a_{jk} (k-1)\cdots (k-i)  t^{k-1-i} = \left(a_{j1},\ldots,a_{jn} \right)\,\left(\begin{array}{c} 0\\ \vdots \\ 0\\ i! \\ \vdots \\  (n-1)^{\underline{i}}\,t^{n-1-i} \end{array}\right)
$$
\end{proof}
As a corollary, we have the generalization of Theorem \ref{kernelmonomial}.
\begin{cor} \label{nullspacegeneral}The null space of $\mathbf{ Bez}_\Phi(P,Q)$ is spanned by the columns of the matrix
$$
\left(X_1^\Phi, X_2^\Phi, \ldots, X_k^\Phi\right),
$$
where each block $X_j^\Phi$ corresponds to a different common root of
$P(t)$ and $Q(t)$. The dimension of each block is the geometric
multiplicity, $k_j,$ of the common root~$x_j$ (i.e., its multiplicity
as a root of the greatest common divisor of $P(t)$ and $Q(t)$).
Moreover each block can be parameterized by the common root~$x_j$
in the form
\begin{equation}
X_j^\Phi=\left( \begin{array}{cccc}
\phi_1(x_j) & \phi_1^{(1)}(x_j) &   \ldots  & \phi^{(k_j-1)}_1(x_j)  \\
\phi_2(x_j) & \phi_2^{(1)}(x_j) &   \ldots  & \phi^{(k_j-1)}_2(x_j) \\
\vdots & \vdots &     &   \vdots \\
 \phi_{n-1}(x_j) & \phi^{(1)}_{n-1}(x_j) & \ldots & \phi^{(k_j-1)}_{n-1}(x_j)   \\
\end{array}
\right)
\end{equation}
\end{cor}

\begin{proof}
Since $$\mathbf{ {Bez}}_{\Phi}(P,Q) =  \mathbf{P}_{ \mathrm{St}\rightarrow \Phi}\, \mathbf{ {Bez}}_{\mathrm{St}}(P,Q) \,\mathbf{P}^t_{\mathrm{St} \rightarrow \Phi},$$ we have

$$ \mathbf{ {Bez}}_{\Phi}(P,Q) \mathbf{P}^t_{  \Phi \rightarrow \mathrm{St}}  =  \mathbf{P}_{ \mathrm{St}\rightarrow \Phi} \mathbf{ {Bez}}_{\mathrm{St}}(P,Q), $$
and the result follows from Theorem \ref{kernelmonomial}  and Lemma \ref{lemagral}.
\end{proof}

As we did before in the monomial basis, the next result provides a closed expression for $\alpha$ if  $\alpha$ is the only simple common root of $P(t)$ and $Q(t)$.

\begin{prop}\label{thm:moment-gen}
Let $P(t)$ and $Q(t)$ be univariate polynomials with only one simple common root $\alpha$. Let $\,1 = a_1\phi_0(t) + \ldots + a_n\phi_{n-1}(t)$, and $\,t = b_1\phi_0(t) + \ldots + b_n\phi_{n-1}(t)$.
If $(u_1, u_2,\dots, u_n)^t$ is a non null vector in the null space of $\bez_\Phi(P,Q)$, then
\begin{equation}\label{solsimple}
\alpha = \frac{ b_1 u_1+\ldots + b_n u_n}{a_1 u_1+\ldots + a_n u_n}.
\end{equation}

\end{prop}
\begin{proof}
By hypothesis, there exists $\lambda\neq 0$ such that
$$
\left(\begin{array}{c}u_1\\ \vdots \\ u_n\end{array}\right) = \lambda\left(\begin{array}{c}\phi_0(\alpha)\\ \vdots \\ \phi_{n-1}(\alpha)\end{array}\right)=   \lambda \mathbf{P}_{\Phi \rightarrow \mathrm{St}}^t \left(\begin{array}{c} 1 \\ \vdots \\\alpha^{n-1} \end{array}\right)
 =  \mathbf{P}_{\Phi \rightarrow \mathrm{St}}^t \left(\begin{array}{c} \lambda \\ \vdots \\ \lambda\alpha^{n-1} \end{array}\right).
$$
Thus,
$$
  \mathbf{P}_{\mathrm{St} \rightarrow \Phi }^t \left(\begin{array}{c}u_1\\ \vdots \\ u_n\end{array}\right) =\left(\begin{array}{c} \lambda \\ \vdots \\ \lambda\alpha^{n-1} \end{array}\right).
$$
Since the first two rows of  $\mathbf{P}_{\mathrm{St} \rightarrow \Phi }^t$ are $(a_1,\ldots,a_n)$ and $(b_1,\ldots,b_n)$ respectively, we have
$
\lambda = a_1u_1+\ldots + a_n u_n$ and $ \lambda\alpha = b_1u_1+\ldots + b_n u_n.$ It follows that
$$
\alpha = \frac{ b_1u_1+\ldots + b_n u_n}{a_1u_1+\ldots + a_n u_n}.
$$

\end{proof}

\begin{example}
We consider Example 3 of \cite{BernsteinBez}. Let $P(t)$ and $Q(t)$ in $\P_4$ be expressed in the Bernstein basis  $\{\beta_i^{(4)}(t) = \binom{4}{i} (1-t)^{4-i} t^i, 0\le i \le 4 \}$ as follows
$$
P(t) = 4\beta_0^{(4)}(t) + 4\beta_1^{(4)}(t) + \frac{19}{6}\beta_2^{(4)}(t) + \frac{3}{2}\beta_3^{(4)}(t);$$ $$
Q(t) = \frac{1}{2}\beta_0^{(4)}(t) + \frac{7}{16}\beta_1^{(4)}(t) + \frac{1}{24}\beta_2^{(4)}(t) - \frac{7}{16}\beta_3^{(4)}(t)-\frac{3}{4} \beta_4^{(4)}(t);
$$
Then, the Bezout matrix in the Bernstein basis is equal to
$$\mathbf{Bez}_\beta(P,Q)= \left[ \begin {array}{cccc} 1&{\frac {17}{6}}&\frac{10}{3}&3
\\ \noalign{\medskip}{\frac {17}{6}}&{\frac {157}{36}}&{\frac {83}{18}
}&4\\ \noalign{\medskip}\frac{10}{3}&{\frac {83}{18}}&{\frac {187}{36}}&{
\frac {19}{4}}\\ \noalign{\medskip}3&4&{\frac {19}{4}}&\frac{9}{2}\end {array}
\right]
$$ and its null space is spanned by the vector $(-1,6,-12,8)$. Then they have only one simple common root. Moreover, since $1=\sum\limits_{k=0}^3\beta_k^{(3)}(t)$ and $t=\sum\limits_{k=1}^3\displaystyle\frac{k}{2}\beta_k^{(3)}(t)$, then by Proposition \ref{thm:moment-gen} their common root is equal to 2,
$$2=\frac{\displaystyle\frac{1}{3}\,6-\frac{2}{3}\,12+8}{(-1+6-12+8)}\cdot$$
\end{example}

\section{Greatest common divisors and Bezout matrices}
Let $P(t), Q_1(t),\ldots,Q_r(t)$ be polynomials in $\P_n$ with $n=\deg(P(t))$. Let
\[
\mathbf{B}^P_\mathrm{St}(Q_1,\ldots,Q_r)=\left(
\begin{array}{c} \bez_{\mathrm{St}}(P,Q_1) \\
\vdots  \\ \bez_{\mathrm{St}}(P,Q_r)
\end{array}
\right),\,\,
\mathbf{B}^P_{\Phi}(Q_1,\ldots,Q_r)=\left(
\begin{array}{c} \bez_{\Phi}(P,Q_1) \\
\vdots  \\ \bez_{\Phi}(P,Q_r)
\end{array}
\right).
\]
Then, the aim of this section is to describe two different methods for computing the polynomial $\gcd(P,Q_1,\ldots,Q_r)$ in the monomial basis, from the matrix $\mathbf{B}^P_{\Phi}(Q_1,\ldots,Q_r)$. We need first to introduce the Barnett's method for computing greatest common divisors (for details, see \cite{Barnett} and \cite{Diaz:BGCD:2002}).

 \subsection{Barnett's Method through $\mathbf{ {Bez}_{\mathrm{St}}(P,Q)}$}\label{barnettst}
The following results are the formulation of Barnett's theorems using Bezout matrices.
\begin{theorem} 
The degree of the greatest common divisor of
$P(t),Q_1(t),\ldots,Q_r(t)$ verifies the following formula
\[
\deg(\gcd(P,Q_1,\ldots,Q_r))=n-\hbox{\rm rank}(
\mathbf{B}^P_\mathrm{St}(Q_1,\ldots,Q_r)).
\]
\end{theorem}

\begin{theorem} \label{independencia}
If $c_1,\ldots ,c_n$ are the columns of the matrix
$\mathbf{B}^P_\mathrm{St}(Q_1 ,\ldots , Q_r)$, and its rank is $n-k,$ then the last $n-k$ columns $c_{k+1},\ldots ,c_n$ are linearly independent,
and each $c_i$ ($1\leq i\leq k$) can be written as a linear
combination of  $c_{k+1},\ldots,c_{n}$.
\end{theorem}  

Finally, it is shown how to use the matrix
$\mathbf{B}^P_\mathrm{St}(Q_1,\ldots , Q_r)$ in order to get the coefficients
of the greatest common divisor of $P(t),Q_1(t),\ldots , Q_r(t)$.

\begin{theorem} \label{coeficientes}
Following the same notation as in Theorem \ref{independencia}, if
\[
 c_{k-i}=\sum\limits_{j=k+1}^{n}h_{k-i}^{j}c_{j},\quad i=0,\ldots ,k-1,
\]
 and
\[
\left(
\begin{array}{c} d_{0} \\ {d}_{1} \\ {d}_{2} \\
\vdots \\ {d}_{k}
\end{array}
\right) =d_{0}\left(
\begin{array}{c} 1 \\ h_{k}^{k+1} \\ h_{k-1}^{k+1} \\
\vdots \\ h_{1}^{k+1}
\end{array}
\right) \, ,
\]
then
\[
 D={d}_{0}{t}^{k}+{d}_{1}{t}^{k-1}+\ldots +{d}_{k-1}t+{d}_{k}
\]
is a greatest common divisor for the polynomials $P(t),Q_1(t),\ldots ,Q_r(t)$.
\end{theorem}

The proof of these results can be found in \cite{Diaz:BGCD:2002}. Observe that if $d_0=1$, then the monic greatest common divisor  is equal to
$t^k+h_k^{k+1}{t}^{k-1}+\ldots +{h}_2^{k+1}t+ h_1^{k+1} $.


\subsection{Barnett's Method through $\mathbf{{Bez}_{\Phi}(P,Q)}$} \label{barnettsect}

We are going to generalize in the following lines the results presented in the previous subsection. Since
\[
\mathbf{B}^P_\Phi(Q_1,\ldots , Q_r)=\left(\begin{array}{ccc}
   \mathbf{P}_{\mathrm{St} \rightarrow\Phi}  &  &  \\
     & \ddots &  \\
     &  & \mathbf{P}_{\mathrm{St} \rightarrow\Phi}
\end{array}\right)\mathbf{B}^P_\mathrm{St} (Q_1,\ldots , Q_r) \mathbf{P}_{\mathrm{St}\rightarrow\Phi}^t \, ,
\]
it is obvious that $\hbox{\rm rank}(
\mathbf{B}^P_\mathrm{St}(Q_1,\ldots,Q_r)) = \hbox{\rm rank}(
\mathbf{B}^P_\Phi(Q_1,\ldots,Q_r))$. Moreover, Theorem \ref{independencia} and Theorem \ref{coeficientes} can be reformulated with the matrix
$$
\T=\mathbf{B}^P_\Phi(Q_1,\ldots , Q_r) \mathbf{P}_{\mathrm{St}\rightarrow\Phi}^{-t} =\mathbf{B}^P_\Phi(Q_1,\ldots , Q_r) \mathbf{P}_{\Phi\rightarrow \mathrm{St}}^t
$$ as follows.

\begin{cor} \label{barnettgeneral}
Let $n-k$ be the rank of $\mathbf{B}^P_\Phi(Q_1,\ldots , Q_r)$ and let $t_1,\ldots ,t_n$ be the columns of the matrix
$\T$. Then  if
\[
 t_{k-i}=\sum\limits_{j=k+1}^{n}h_{k-i}^{j}t_{j},\quad i=0,\ldots ,k-1,
\]
 and
\[
\left(
\begin{array}{c} d_{0} \\ {d}_{1} \\ {d}_{2} \\
\vdots \\ {d}_{k}
\end{array}
\right) =d_{0}\left(
\begin{array}{c} 1 \\ h_{k}^{k+1} \\ h_{k-1}^{k+1} \\
\vdots \\ h_{1}^{k+1}
\end{array}
\right) \, ,
\]
then
\[
 D={d}_{0}{t}^{k}+{d}_{1}{t}^{k-1}+\ldots +{d}_{k-1}t+{d}_{k}
\]
is a greatest common divisor for the polynomials $P(t),Q_1(t),\ldots ,Q_r(t)$.
\end{cor}
We would like to remark that Corollary \ref{barnettgeneral} provides a new way to obtain the coefficients of $\gcd(P,Q_1,\ldots,Q_r)$ in the monomial basis from the Bezout matrix in an arbitrary basis.

\subsection{Nullspaces and gcd's} \label{sectionnullspacesgcd}
We give next the other method to compute the greatest common divisor.
\begin{prop}\label{gcdKer}
Suppose  that $\gcd(P,Q_1,\ldots,Q_r)=t^{k}+{d}_{1}{t}^{k-1}+\ldots +{d}_{k}$. Let  $\mathbf{N}\in\mathcal M_{n\times k}(\mathbb K)$ be a matrix whose columns form a basis of the null space of $\mathbf{B}^P_\mathrm{St}(Q_1,\ldots,Q_r)$, and let $\mathbf{Z}$ be the submatrix of $\mathbf{N}$ defined by the first $k+1$ rows.
Then we have
\begin{equation}\label{sistemagral}
\left( d_k  ,\ldots     ,d_1,1 \right) \mathbf{Z} = (0,\ldots,0),
\end{equation}
\end{prop}
\begin{proof} 
Observe that the linear combinations introduced in Theorem \ref{coeficientes} define  a basis of the null space of $\mathbf{B}^P_\mathrm{St}(Q_1,\ldots,Q_r)$.
More specifically and following the notation of Theorem \ref{coeficientes}, the columns of the following triangular matrix $\F$ define a basis of the null space of $\mathbf{B}^P_\mathrm{St}(Q_1,\ldots,Q_r)$,
$$
\F=\left(\begin{array}{cccc}
-1&   &   &   \\
0 & -1 &   &   \\
\vdots&\vdots   &  \ddots &   \\
0&  0 &    &  -1 \\
h_1^{k+1}& h_2^{k+1}  &   &   h_k^{k+1}\\
\vdots& \vdots   &   &  \vdots\\
h_1^n& h_2^{n}  &   & h_k^n
\end{array}\right)\, .
$$
Obviously,
\begin{equation}
\left( h_1^{k+1}  ,\ldots     ,h_k^{k+1}, 1 \right)  \left(\begin{array}{cccc}
-1&   &   &   \\
0 & -1 &   &   \\
\vdots&\vdots   &  \ddots &   \\
0&  0 &    &  -1 \\
h_1^{k+1}& h_2^{k+1}  &   &   h_k^{k+1}
\end{array}\right) = (0,\ldots,0).
\end{equation}
On the other hand, if $\mathbf{N}$ is a matrix whose columns form a basis of the null space of $\mathbf{B}^P_\mathrm{St}(Q_1,\ldots,Q_r)$, then there exists a nonsingular matrix $\mathbf{S}\in \mathcal M_{k\times k}(\mathbb R)$ such that  $\mathbf{N}=\mathbf{F\,S}$. Therefore, if $\mathbf{Z}$ is the submatrix of $\mathbf{N}$ defined by the first $k+1$ rows, then
\begin{equation} \label{firskrowsofZindependientes}
\mathbf{Z}=\left(\begin{array}{cccc}
-1&   &   &   \\
0 & -1 &   &   \\
\vdots&\vdots   &  \ddots &   \\
0&  0 &    &  -1 \\
h_1^{k+1}& h_2^{k+1}  &   &   h_k^{k+1}
\end{array}\right)  \mathbf{S}.
\end{equation}
Thus $\left( h_1^{k+1}  ,\ldots     ,h_k^{k+1}, 1 \right) \mathbf{Z} = (0,\ldots,0)$. Since $\left( h_1^{k+1}  ,\ldots     ,h_k^{k+1}, 1 \right) =\left( d_k  ,\ldots     ,d_1,1 \right) $, the proposition is proved.
\end{proof}
Hence, the greatest common divisor of several polynomials can be computed by solving the linear system (\ref{sistemagral}). The generalization of  Proposition \ref{gcdKer} is as follows.
\begin{cor}\label{nul} Let  $\mathbf{N}^\Phi\in\mathcal M_{n\times k}(\mathbb K)$ be a matrix whose columns define a basis of the null space $\mathbf{B}^P_\Phi(Q_1,\ldots,Q_r)$.
Thus, if $\mathbf{Z^\Phi}$ denotes the matrix defined by the first $k+1$ rows of $\mathbf{P}^t_{\mathrm{St} \rightarrow \Phi} \mathbf{N^\Phi}$, then
\begin{equation}\label{sistemaphi}
\left( d_k  ,\ldots     ,d_1,1 \right) \mathbf{Z^\Phi} = (0,\ldots,0).
\end{equation}
\end{cor}

 \begin{example} Supposse that we have two polynomials $P(t)$ and $Q(t)$ in $\P_4$ expressed in the Hermite interpolation basis denoted by $\mathbf{H}.$ i.e., the polynomials are given by values. As input, the nodes are
$\boldsymbol{\tau}=[-1,3,4]$,
and the confluencies are $ \boldsymbol{s}=[2,2,1]$. The values of $P(t)$ and $Q(t)$ together with their derivatives are
$$\boldsymbol{p}=[6, -11, 26, 53, 126],\quad  \boldsymbol{q}=[-12, 16, 4, 8, 18].$$
In this case,  the confluent Vandermonde matrix is the transformation matrix $\mathbf{P}_{\mathrm{St} \rightarrow \mathbf{H}} $, the confluent Bezout matrix (in \cite{eaca14Hermite} we describe how to directly compute it from the values) is  given by
$$\mathbf{Bez}_\mathbf{H}(P,Q)=
\left[ \begin {array}{cccc} 36&-102&-84&-150\\ \noalign{\medskip}-102
&181&94&173\\ \noalign{\medskip}-84&94&4&14\\ \noalign{\medskip}-150&
173&14&37\end {array} \right] ,
$$
and the null space is spanned by the vectors $ ( -22,-21,0,9)$ and $(-13,-12,9,0)$. Following the notation of Corollary \ref{nul}, $k=2$  and the matrix $\mathbf{Z^\mathbf{H}}$ is given by
$$
\mathbf{Z^\mathbf{H}}=\left[ \begin {array}{cc} -22&-4\\ \noalign{\medskip}10&28
\\ \noalign{\medskip}74&92\end {array} \right].
$$
By solving the linear system equations (\ref{sistemaphi}), we obtain that $\left( d_2   ,d_1,1 \right) =(2,-3,1)$. Thus the greatest common divisor of $P(t)$ and $Q(t)$ is equal to ${x}^{2}-3\,x+2$. Observe that we have not had to convert the polynomials to the monomial basis at any moment.
\end{example}

\section{Conclusions and future work} \label{conclusions}
Given a set of polynomials $P(t),Q_1(t), \ldots, Q_r(t)$ represented in a polynomial basis $\Phi$, this manuscript introduces, on the one hand, the structure of the null space of the Bezout matrix $\mathbf{B}^P_\Phi(Q_1(t), \ldots, Q_r(t)) $, easily deduced from Theorem \ref{nullspacegeneral}. 

On the other hand,  Corollary \ref{barnettgeneral} in Section \ref{barnettsect} and Corollary \ref{nul} in Section \ref{sectionnullspacesgcd} present two different methods for computing their gcd from $\mathbf{B}^P_\Phi(Q_1(t) \ldots, Q_r(t))$. We estimate that, in most cases, this methodology is better than converting first the polynomials to the monomial basis, and then computing their gcd. For example, it is known that when the data is given by values (Lagrange interpolation data), working directly in the Lagrange basis is usually better than working in other basis, even in the Bernstein one (see for example \cite{Berrut:BLI:2004, Corless:LBB:2004, shakoori:2007}).  We are currently comparing our methodology with others (see for example \cite{BernsteinBez} for the Bernstein basis and \cite{Cheng:2009:CPL:1504347.1504353} for the Lagrange Basis), analyzing the numerical behaviour for each basis.


\section{Acknowledgments}
The authors are partially supported by the Spanish ``Ministerio de Econom\'ia y Competitividad" and by the European Regional Development Fund (ERDF), under the Project MTM2011-25816-C02-02.

\bibliographystyle{plain}
\bibliography{bez_lin}

\begin{thebibliography}{10}

\bibitem{Joua}
F.~Ap{\'e}ry and J.~P. Jouanolou.
\newblock {\em R{\'e}sultant et sous-r{\'e}sultants : le cas d'une variable :
  avec exercices corrig{\'e}s}.
\newblock Hermann, Paris, 2006.

\bibitem{Barnett}
S.~Barnett.
\newblock Greatest common divisor of several polynomials.
\newblock {\em Proceedings of the Cambridge Philosophical Society},
  70:263--268, 1971.

\bibitem{barnett:83}
S.~Barnett.
\newblock {\em Polynomials and Linear Control Systems}.
\newblock Marcel Dekker, 1983.

\bibitem{ChebyBez}
S.~Barnett.
\newblock A {B}ezoutian matrix for {C}hebyshev polynomials.
\newblock In {\em Application of Matrix Theory}, volume~22, pages 137--149, New
  York, 1989. The Clarendon Press.

\bibitem{MFbook}
S.~Basu, R.~Pollack, and M.F. Roy.
\newblock {\em Algorithms in real algebraic geometry, 2nd edition}.
\newblock Springer Verlag, Berlin, Germany, 2006.

\bibitem{DBLP:conf/alenex/BerberichES11}
Eric Berberich, Pavel Emeliyanenko, and Michael Sagraloff.
\newblock An elimination method for solving bivariate polynomial systems:
  Eliminating the usual drawbacks.
\newblock In Matthias M{\"u}ller-Hannemann and Renato Fonseca~F. Werneck,
  editors, {\em ALENEX}, pages 35--47. SIAM, 2011.

\bibitem{Berrut:BLI:2004}
J.P. Berrut and L.N. Trefethen.
\newblock Barycentric {L}agrange interpolation.
\newblock {\em SIAM Review}, 46(3):501--517, 2004.

\bibitem{BernsteinBez}
D.~A. Bini and L.~Gemignani.
\newblock Bernstein-{B}ezoutian matrices.
\newblock {\em Theoretical Computer Science}, 315(2--3):319--333, 2004.

\bibitem{BiniPan:1994}
Dario Bini and Victor Pan.
\newblock {\em Polynomial and Matrix Computations}.
\newblock Birkh\"auser, 1994.

\bibitem{Cheng:2009:CPL:1504347.1504353}
Howard Cheng, George Labahn, and Wei Zhou.
\newblock Computing polynomial lcm and gcd in lagrange basis.
\newblock {\em ACM Commun. Comput. Algebra}, 42(3):129--130, February 2009.

\bibitem{Corless:LBB:2004}
R.~M. Corless and Stephen~M. Watt.
\newblock Bernstein bases are optimal, but, sometimes, {L}agrange bases are
  better.
\newblock In {\em Proceedings of SYNASC, Timisoara}, pages 141--153. MIRTON
  Press, September 2004.

\bibitem{D-TFG-VS:2012}
G.~M. Diaz-Toca, M.~Fioravanti, L.~Gonzalez-Vega, and A.~Shakoori.
\newblock Using implicit equations of parametric curves and surfaces without
  computing them: Polynomial algebra by values.
\newblock {\em Computer Aided Geometric Design}, 30(1):116--139, January 2013.

\bibitem{Diaz:BGCD:2002}
G.M. Diaz-Toca and L.~Gonzalez-Vega.
\newblock Barnett's theorems about the greatest common divisor of several
  univariate polynomials through {B}ezout-like matrices.
\newblock {\em Journal of Symbolic Computation}, 34(1):59--81, 2002.

\bibitem{FR:1996}
P.A. Fuhrmann.
\newblock {\em A Polynomial Approach to Linear Algebra}.
\newblock Springer Verlag, New York, USA, 1996.

\bibitem{HeinigRost}
G.~Heinig and K.~Rost.
\newblock Algebraic methods for toeplitz-like matrices and operators.
\newblock {\em Operator Theory: Advances and Applications}, 13, 1984.

\bibitem{HF}
U.~Helmke and P.~A Fuhrmann.
\newblock Bezoutians.
\newblock {\em Linear Algebra and Its Applications}, 122/123/124:1039--1097,
  1989.

\bibitem{LT}
P.~Lancaster and M.~Tismenetsky.
\newblock {\em The theory of matrices}.
\newblock Computer Science and Applied Mathematics. Academic Press, USA, 1985.

\bibitem{mignotte}
M.~Mignotte.
\newblock {\em Mathematics for Computer Algebra}.
\newblock Springer Verlag, New York, USA, 1992.

\bibitem{Shakoori:BML:2004}
A.~Shakoori.
\newblock The {B}\'ezout matrix in the {L}agrange basis.
\newblock In Laureano Gonzalez-Vega and Tomas Recio, editors, {\em Proceedings
  EACA}, pages 295--299, June 2004.

\bibitem{shakoori:2007}
A.~Shakoori.
\newblock {\em Bivariate Polynomial Solver by Values}.
\newblock PhD thesis, The University of Western Ontario, 2007.

\bibitem{Yang:2001}
Zheng-Hong Yang.
\newblock Polynomial bezoutian matrix with respect to a general basis.
\newblock {\em Linear Algebra and Its Applications}, 331:165--179, 2001.

\end{thebibliography}

\end{document}